\documentclass[11pt]{article}

\usepackage{amsmath,amsfonts,amssymb,amsthm,latexsym}

\newcommand\qbi[3]{{{#1}\atopwithdelims[]{#2}}_{#3}}
\newcommand\bi[2]{{{#1}\atopwithdelims(){#2}}}

\newcommand{\la}{\lambda}
\def\C{\mathbb{ C}} \def\N{\mathbb{ N}}
 \def\Z{\mathbb{Z}}
\def\Q{\mathbb{Q}}

\newtheorem{theo}{Theorem}[section]
\newtheorem{prop}[theo]{Proposition}
\newtheorem{coro}[theo]{Corollary}
\newtheorem{conj}[theo]{Conjecture}
\newtheorem{lem}[theo]{Lemma}
\newtheorem{rem}[theo]{Remark}
\newenvironment{abstr}{\noindent{\bf Abstract. }}{\smallskip}
\numberwithin{equation}{section}

\title{Diophantine properties for $q$-analogues of Dirichlet's beta function at positive integers}
\author{Fr\'ed\'eric Jouhet and Elie Mosaki}
\date{}

\begin{document}
\maketitle

\vspace{0.5 cm}

\noindent {\small \emph{2000 Mathematics Subject Classification}: Primary 11J72; Secondary 11M36; 33D15.}

\vspace{0.7 cm}

\begin{abstr}
\small In this paper, we define $q$-analogues of Dirichlet's beta function at positive integers, which can be written as $\beta_q(s)=\sum_{k\geq1}\sum_{d|k}\chi(k/d)d^{s-1}q^k$ for $s\in\N^*$, where $q$ is a complex number such that $|q|<1$ and $\chi$ is the non trivial Dirichlet character modulo $4$. For odd $s$, these expressions are connected with the automorphic world, in particular with Eisenstein series of level $4$. From this, we derive through Nesterenko's work the transcendance of the numbers $\beta_q(2s+1)$ for $q$ algebraic such that $0<|q|<1$. Our main result concerns the nature of the numbers $\beta_q(2s)$: we give a lower bound for the dimension of the vector space over $\Q$ spanned by
$1,\beta_q(2),\beta_q(4),\dots,\beta_q(A)$, where $1/q\in\Z\setminus\{-1;1\}$ and $A$ is an even integer. As consequences, for $1/q\in\Z\setminus\{-1;1\}$, on the one hand there is an infinity of irrational numbers among $\beta_q(2),\beta_q(4),\dots$, and on the other hand at least one
of the numbers $\beta_q(2),\beta_q(4),\dots, \beta_q(20)$ is irrational.
\end{abstr}

\vspace{0.5 cm}

\section{Introduction}

For any complex number $s$ with $Re(s)\geq1$, Dirichlet's beta function at $s$ is defined by:
$$\beta(s)=\sum_{k=0}^\infty\frac{(-1)^k}{(2k+1)^s}\cdot$$
Recall Euler's identity:
$$\beta(2m+1)=\frac{(-1)^mE_{2m}}{2^{2m+2}(2m)!}\pi^{2m+1},$$
where  $m\in\N$ and the rational numbers $E_{2m}$ are Euler numbers defined by $1/\cosh(z)=\sum_{k\geq0}E_kz^k/k!$. Thus, as for the values of Riemann's zeta function at odd positive integers, Lindemann's Theorem yields that \emph{for $m\in\N$, $\beta(2m+1)$ is a transcendental number}. However, nothing similar can be said concerning the values at even positive integers; the best known result in that direction is due to Rivoal and Zudilin \cite{RZ}: \emph{at least one of the numbers $\beta(2),\beta(4),\dots,\beta(12)$ is irrational}. \\

In this article, we define $q$-analogues of the values of $\beta$ at positive integers. They can be written for $s\in\N^*$ and for any complex number $q$ with $|q|<1$:
\begin{equation}\label{betaq}
\beta_q(s):=\sum_{k\geq1}k^{s-1}\frac{q^k}{1+q^{2k}}=\sum_{k\geq1}\sum_{d|k}\chi(k/d)d^{s-1}q^k,
\end{equation}
where $\chi$ is the non trivial Dirichlet character modulo~4, defined by $\chi(2m+1)=(-1)^m$ and $\chi(2m)=0$. One can justify the term of $q$-analogue by the following relation, valid for $s\in\N^*$ (see the last section of this paper for a proof):
$$\lim_{q\to 1}(1-q)^s\beta_q(s)=(s-1)!\beta(s).$$

Similarly to the $q$-analogues of Riemann's zeta function at even positive integers considered in \cite{KRZ, JM}, our definition \eqref{betaq} is related to modular forms when $s$ in an odd positive integer. Indeed, consider the case $s=1$, for which the first expression of \eqref{betaq} can be written: 
\begin{equation}\label{betaq1}
\beta_q(1)=\sum_{k\geq1}\sum_{m\geq0}(-1)^mq^{(2m+1)k}=\sum_{m\geq0}(-1)^m\frac{q^{2m+1}}{1-q^{2m+1}}\cdot
\end{equation}
This yields immediately $\beta_q(1)=(\pi_{q}-1)/4$, where $\pi_q$ is a $q$-analogue of $\pi$. The series $\pi_q$ is considered in \cite{BZ, BZ2}, where an upper bound for its irrationality exponent is given. We also have (see \cite{BZ}):
$$\beta_q(1)=\sum_{k\geq1}q^k\sum_{d|k}\chi(d)=\frac{\theta^2(q)-1}{4},$$
where $\theta(q):=\sum_{n\in\mathbb{Z}}q^{n^2}$ is the classical theta function. This shows that if we set $q=\mbox{e}^{2i\pi z}$, $\beta_q(1)$ is, up to a rational constant, the Fourier expansion of a weight $1$ modular form on $\Gamma_1(4)$ \cite[p. 138, Proposition 30]{Ko}. Moreover, as remarked in \cite{BZ}, Nesterenko's algebraic independance Theorem from \cite{Ne2} shows that $\theta(q)$, and therefore $\beta_q(1)$, is a transcendental number when $q$ is algebraic such that $0<|q|<1$.

\noindent Concerning the other values at odd positive integers, one sees that for $s\geq1$ and $q=\mbox{e}^{2i\pi z}$, $\beta_q(2s+1)$ is also the Fourier expansion, with algebraic coefficients, of a weight $2s+1$ modular form  on $\Gamma_1(4)$. More precisely, consider the level~$4$ Eisenstein series \cite[p. 131]{Ko}:
$$G_{2s+1}^{(1,0)}(z):=\sum_{{(m_1,m_2)\in\mathbb{Z}^2\atop (m_1,m_2)\equiv(1,0)\,\mbox{{\scriptsize mod}}\,4}}\frac{1}{(m_1z+m_2)^{2s+1}}\cdot$$
Then we have the following Fourier expansion \cite[proposition 22]{Ko}:
$$G_{2s+1}^{(1,0)}(z)=\frac{i}{(2s)!}\left(\frac{\pi}{2}\right)^{2s+1}\sum_{k\geq1}\sum_{d|k}\chi(k/d)d^{2s}q^{k/4}.$$
Hence
$$\beta_q(2s+1)=i(-1)^{s+1}\frac{E_{2s}}{2\beta(2s+1)}\,G_{2s+1}^{(1,0)}(4z),$$
and it is not difficult to see that $z\mapsto G_{2s+1}^{(1,0)}(4z)$ is a weight $2s+1$ modular form on $\Gamma_1(4)$. Set $\phi_s(q):=\beta_q(2s+1)/\theta^{4s+2}(q)$, which is a modular function (with weight 0) on $\Gamma_1(4)$, having a Fourier expansion with algebraic coefficients. Thus $\phi_s$ is algebraic on $\mathbb{Q}(J)$ (see \cite[p. 144, Problem 7]{Ko}), where $J$ is the modular invariant. Assume first that for $s\geq1$, $\phi_s$ is not a constant. We can deduce that $J(q)$ and $\phi_s(q)$ are algebraically dependant. Assume from now on that $q$ is algebraic such that $0<|q|<1$. Nesterenko's algebraic independance Theorem from \cite{Ne2} shows that $J(q)$ and $\theta(q)$ are algebraically independant. Therefore $\beta_q(2s+1)$ and $\theta(q)$ are necessarily also algebraically independant, and in particular $\beta_q(2s+1)$ is a transcendental number. If now $\phi_s$ is a constant (necessarily algebraic), then $\theta^{4s+2}(q)$ is a transcendental number by \cite{Ne2}, and so $\beta_q(2s+1)$ is again a transcendental number. To summarize, we have the following result, which can be compared to the transcendence of the values of Dirichlet's beta function at odd positive integers, as well as  the transcendence of the values of $\zeta_q$ at even positive integers in \cite{KRZ}: \emph{for $s\in\N$ and $q$ algebraic such that $0<|q|<1$, $\beta_q(2s+1)$ is a transcendental number}.\\

Consider now the values at even positive integers $\beta_q(2s)$, which do not seem to be directly related to Eisenstein series. We will prove the following Theorem, which is the main result of the present paper:
\begin{theo}
For $1/q\in\Z\setminus\{-1;1\}$ and any odd integer $A\geq 3$, we have the following lower bound:
\begin{equation}\label{efminor}
\dim_{\Q}\left(\Q+\Q\beta_q(2)+\dots+\Q\beta_q(A-1)\right)\geq
f(A),
\end{equation}
where
$$
f(A)=\max_{r\in\N\atop 1\leq r<A/2}f(r;A)\;\;\;\;\mbox{and}\;\;f(r;A):=\frac{4rA+A-4r^2}{\left(\frac{48}{\pi^2}+2\right)A+8r^2-\frac{16}{\pi^2}+\frac{16r}{3}}\cdot
$$
Moreover $f(A)$ satisfies $\displaystyle \;f(A)\sim\frac{\pi}{2\sqrt{\pi^2+24}}\,\sqrt A\; $ when $A\to+\infty$.
\end{theo}
The previous asymptotic estimate for $f(A)$ gives immediately the following:
\begin{coro}
For $1/q\in\Z\setminus\{-1;1\}$, there are infinitely many irrational numbers among $\beta_q(2)$, $\beta_q(4)$, $\beta_q(6)$,\dots 
\end{coro}
On the other hand, the estimate $f(3;21)\geq 1,02...$ gives the following quantitative version:
\begin{coro}
For $1/q\in\Z\setminus\{-1;1\}$, at least one of the numbers $\beta_q(2),\beta_q(4)$, $\beta_q(6)$,$\dots,\beta_q(20)$ is irrational.
\end{coro}
Now we aim to sketch the proof of Theorem~1.1, which will be more detailed later on. We need the following Proposition, which is a special case of Nesterenko's linear independance criterion from \cite{Ne}:
\begin{prop}
Let $N\geq 2$ be an integer, and $v_1,\dots,v_N$ be real numbers.
Assume that there exist $N$ integer sequences $(p_{j,n})_{n\geq0}$
and two real numbers $\alpha_1$ et $\alpha_2$
with $\alpha_2>0$ such that:\\
i)
$\displaystyle\lim_{n\to+\infty}\frac{1}{n^2}\log|p_{1,n}v_1+\dots+p_{N,n}v_N|=-\alpha_1$,\\
ii) for all $j\in\{1,\dots,N\}$, we have
$\displaystyle\limsup_{n\to+\infty}\frac{1}{n^2}\log|p_{j,n}|\leq\alpha_2$.\\
Then the dimension of the $\Q$-vector space spanned by $v_1,\dots,v_N$ satisfies:
$$\dim_{\Q}\left(\Q v_1+\dots+\Q v_N\right)\geq 1+\frac{\alpha_1}{\alpha_2}\cdot$$
\end{prop}
\medskip
In order to use Nesterenko's criterion in our context, we shall study the following hypergeometric series (see section~2 for the notations):
\begin{multline*}
S_n(q):=(q)_n^{A-2r}\sum_{k\geq
1}(-1)^{k+1}q^{(k-1/2)((A-2r)n/2+A/2-1)}\\
\times(1-q^{2k+n-1})\frac{(q^{k-rn},q^{k+n})_{rn}}{(q^{k-1/2})_{n+1}^{A}},
\end{multline*}
where $A$ is an integer, $r\in\N^*$ and $A-2r>0$. In a first step, we rewrite $S_n(q^2)$ as a linear combination of some $\beta_q(2m)$, $m\in\N^*$:
$$S_n(q^2)=\hat{P}_{0,n}(q^2)+\sum_{{j=2\atop j\,\mbox{\scriptsize
{even}}}}^{A-1}\hat{P}_{j,n}(q^2)\beta_q(j),$$ where $|q|<1$, $A$ and $n$ are \emph{odd} positive integers and $\hat{P}_{j,n}(q^2)$ are \emph{a priori} in $\Q(q)$,
i.e. rational fractions in the variable $q$ (thus in the variable $1/q$) with coefficients in $\Q$. In a second step, we look for a common denominator $D_n(q)$ to these rational fractions in the variable $1/q$, satisfying:
$$D_n(q)\hat{P}_{j,n}(q^2)\in\Z\left[\frac{1}{q}\right]\;\;\;\forall
j\in\{0,2,4,\dots,A-1\}.$$
Next, we shall prove the following asymptotic estimates which are true for all $0<|q|<1$:
\begin{eqnarray*}
\lim_{n\to+\infty}\frac{1}{n^2}\log|S_n(q)|&=&-\frac{1}{2}r(A-2r)\log|1/q|,\\
\limsup_{{n\to+\infty\atop n\,\mbox{\scriptsize{odd}}}}\frac{1}{n^2}\log|\hat{P}_{j,n}(q)|&\leq&
\frac{1}{8}(A+4r^2)\log|1/q|,\\
\lim_{n\to+\infty}\frac{1}{n^2}\log|D_n(q)|&=&\left(\frac{A}{4}+r^2+\frac{12}{\pi^2}(A-1)+\frac{4r}{3}+\frac{8}{\pi^2}\right)\log|1/q|.
\end{eqnarray*}
Assume $1/q\in\Z\setminus\{-1;1\}$. Then we apply Proposition~1.4 to the $N=(A+1)/2$ integer sequences  $(D_n(q)\times \hat{P}_{j,n}(q^2))_{n\,\mbox{\scriptsize{odd}}}$, which yields \eqref{efminor}.
   
\vspace{0.2cm}
\noindent The estimate $f(A)\sim\pi\sqrt A/2\sqrt{\pi^2+24}$ for $A\to+\infty$, is obtained by choosing $r=u\sqrt{A}$ and finding the maximal value of $f(u\sqrt{A};A)/\sqrt{A}$ in the variable $u$.\\
\begin{rem}
\emph{Corollary~1.3 can be proved direcly without Nesterenko's criterion. Indeed, for $1/q\in\Z\setminus\{-1;1\}$, it is enough to obtain an asymptotic estimate of the linear combination $D_n(q)\times S_n(q^2)$ in $\beta_q(2)$, $\beta_q(4)$, $\dots,\beta_q(20)$ with integer coefficients (by choosing $A=21$ and $r=3$). Therefore it is not necessary to find an upper bound for the height of the coefficients of the linear form, this is only useful for linear independance.}
\end{rem}
\medskip
This article is organized as follows. In section~2 we recall a few notations on $q$-series, which will be useful later. The third section is concerned with the study of the series $S_n(q)$ mentioned before. In particular, the structure of this series (it is a \emph{very-well-poised} basic hypergeometric series) yields interesting properties. In particular, when $n$ is an odd integer, an appropriate expansion of $S_n(q^2)$ will give explicitely the already mentioned linear combination and its coefficients $\hat{P}_{j,n}(q^2)$. In the fourth section we find a common denominator $D_n(q)$ to the coefficients $\hat{P}_{j,n}(q^2)$, by using arithmetical techniques and cyclotomic polynomials. In section~5, we study the asymptotics of $S_n(q)$, $\hat{P}_{j,n}(q)$ through Cauchy's formula, and $D_n(q)$ by using the M\"obius inversion. Finally, in the sixth and last section, we establish some links between $\beta_q$ and $\beta$, and we end by stating a $q$-denominators Conjecture.

\section{Notations}

We recall some standard definitions and notations for $q$-series, which can be found in \cite{GR}. 

\medskip
Let $q$ be a fixed complex parameter (the ``base'') with $|q|\neq 1$. We define for any real number $a$ and any $k\in\N$ the \emph{$q$-shifted factorial} by:
\begin{equation*}
(a)_k\equiv
(a;q)_k:=\left\{\begin{array}{l}1\;\;\mbox{if}\;\;k=0\\(1-a)\dots
(1-aq^{k-1})\;\;\mbox{if}\;\;k>0.
\end{array}\right.
\end{equation*}
The base $q$ can be omitted when there is no confusion (writing $(a)_k$ for $(a;q)_k$, etc). For the sake of simplicity, write for $k\in\N$:
\begin{equation*}
(a_1,\ldots,a_m)_k:=(a_1)_k\times\cdots\times(a_m)_k.
\end{equation*}
\medskip
Recall the definition of the \emph{$q$-binomial coefficient}:
$$\qbi{n}{k}{q}:=\frac{(q)_n}{(q)_k(q)_{n-k}},$$
which is a polynomial in the variable $q$, with integer coefficients (see for example \cite{St}).

\medskip
\noindent Further, recall the definition of the \emph{basic hypergeometric series ${}_{s+1}\phi_s$}:
\begin{equation*}
{}_{s+1}\phi_s\!\left[\begin{matrix}a_0,a_1,\dots,a_s\\
b_1,\dots,b_{s}\end{matrix};q,z\right]:=
\sum_{k=0}^\infty\frac{(a_0,a_1,\dots,a_s)_k}{(q,b_1,\dots,b_{s})_k}z^k,
\end{equation*}
with $a_j\in\C$ for $0\leq j\leq s$, and $b_j\,q^k\neq 1$ for all $k\in\N$ and $1\leq j\leq s$. The series converges for $|q|<1$ and $|z|<1$, or  $|q|>1$ and $|z|<|b_1\dots b_s/a_1\dots a_{s+1}|$, and we say that ${}_{s+1}\phi_s$ is:
\begin{itemize}
\item \emph{well-poised} if $qa_0=a_1b_1=\dots=a_sb_s$ 
\item \emph{very-well-poised} if it is well-poised and $a_1=q\sqrt{a_0}=-a_2$.
\end{itemize}

\section{A very-well-poised series}

We consider the following basic hypergeometric series:
\begin{multline}\label{Sn}
S_n(q):=(q)_n^{A-2r}\sum_{k\geq
1}(-1)^{k+1}q^{(k-1/2)((A-2r)n/2+A/2-1)}\\
\times(1-q^{2k+n-1})\frac{(q^{k-rn},q^{k+n})_{rn}}{(q^{k-1/2})_{n+1}^{A}},
\end{multline}
for any odd integer $A$, $r\in\N^*$ and $A-2r>0$. Note that this series converges for all $|q|\neq1$. The series $S_n(q)$ satisfies:
\begin{equation}\label{tbe}
S_n(1/q)=q^{n(r-1)}S_n(q),
\end{equation}
which comes from the choice of the power of $q$ inside the sum of \eqref{Sn}. Thanks to \eqref{tbe}, it will be possible to expand $S_n(q^2)$ as a linear combination over $\Q(q)$ of the values of $\beta_q$ at even positive integers only, although one would expect values at odd positive integers as well. Besides, notice that:
\begin{multline*}
S_n(q)=(-1)^{rn}q^{(rn+1/2)((A-2r)n/2+A/2-1)}\\
\times(1-q^{(2r+1)n+1})(q)_n^{A-2r}\frac{(q,q^{(r+1)n+1})_{rn}}{(q^{rn+1/2})_{n+1}^{A}}\\
\times{}_{A+4}\phi_{A+3}\!\left[\begin{matrix}a,q\sqrt a,-q\sqrt
a,q^{rn+1},q^{rn+1/2}\dots,q^{rn+1/2}\\
\sqrt a,-\sqrt
a,q^{(r+1)n+1},q^{(r+1)n+3/2}\dots,q^{(r+1)n+3/2}\end{matrix};q,z\right],
\end{multline*}
with $a=q^{(2r+1)n+1}$ and $z=-q^{(A-2r)n/2+A/2-1}$. This shows that $S_n(q)$ is a very-well-poised basic hypergeometric series. 

\subsection{Some auxiliary functions}

For all $|q|\neq1$ and $s\in\N^*$, we consider the functions:
\begin{equation}\label{Ys}
Y_s(q):=\sum_{k\geq0}(-1)^{k}\frac{q^{2k+1}}{(1-q^{2k+1})^s}\cdot
\end{equation}
We will need the \emph{signless Stirling numbers of the first kind} (see \cite{St}), which are integers denoted by $c(s,j)$ (where $s$ and $j$ are two integers such that $1\leq j\leq s$) and that are defined by:
$$(x)_s:=x(x+1)\dots(x+s-1)=\sum_{j=1}^sc(s,j)x^j.$$
The following result gives for $s\geq2$ the expansion of the functions $Y_s$ in terms of values of $\beta_q$ at positive integers: 
\begin{lem}
For all $|q|<1$ and any integer $s\geq2$:
\begin{equation}\label{Ysbetaq}
Y_s(q)=\frac{1}{(s-1)!}\sum_{j=2}^sc(s-1,j-1)\beta_q(j).
\end{equation}
\end{lem}
\begin{proof}
The definition \eqref{Ys} can be expanded as follows, for $|q|<1$:
\begin{eqnarray}
Y_s(q)&=&\frac{1}{(s-1)!}\sum_{k\geq0}\sum_{l\geq1}(-1)^{k}(l)_{s-1}q^{(2k+1)l}\label{ys(q)}\\
&=&\frac{1}{(s-1)!}\sum_{k\geq0}\sum_{l\geq1}(-1)^{k}\sum_{j=2}^sc(s-1,j-1)l^{j-1}q^{(2k+1)l}\nonumber\\
&=&\frac{1}{(s-1)!}\sum_{j=2}^sc(s-1,j-1)\sum_{l\geq1}\frac{q^l}{1+q^{2l}},\nonumber
\end{eqnarray}
and we conclude by using the definition \eqref{betaq}.
\end{proof}
The following Lemma gives for $s\geq2$ the expansion of the functions $Y_s$ in terms of values of $\beta_q$ at \emph{even} positive integers only: 
\begin{lem}
For all $0<|q|<1$ and any integer $s\geq2$:
\begin{equation}\label{Ysbetaqpair}
Y_s(q)+Y_s(1/q)=\frac{2}{(s-1)!}\sum_{{j=2\atop j\mbox{\scriptsize even}}}^sc(s-1,j-1)\beta_q(j).
\end{equation}
\end{lem}
\begin{proof}
Set $0<|q|<1$. Starting from the definition \eqref{Ys}, we get:
\begin{eqnarray}
Y_s(1/q)&=&(-1)^s\sum_{k\geq0}(-1)^{k}\frac{q^{(2k+1)(s-1)}}{(1-q^{2k+1})^s}\nonumber\\
&=&(-1)^s\sum_{k\geq0}\sum_{l\geq1}(-1)^{k}q^{(2k+1)(s-2)}\frac{(l)_{s-1}}{(s-1)!}q^{(2k+1)l}\nonumber\\
&=&\frac{(-1)^s}{(s-1)!}\sum_{k\geq0}\sum_{l\geq1}(-1)^{k}(l-s+2)_{s-1}q^{(2k+1)l}\label{ys(1/q)}.
\end{eqnarray}
Then we collect expressions \eqref{ys(q)} and \eqref{ys(1/q)}, invert summations, and sum over $k$:
\begin{eqnarray*}
Y_s(q)+Y_s(1/q)&=&\frac{1}{(s-1)!}\sum_{l\geq1}\left[(l)_{s-1}-(-l)_{s-1}\right]\frac{q^l}{1+q^{2l}}\\
&=&\frac{1}{(s-1)!}\sum_{j=2}^sc(s-1,j-1)(1+(-1)^j)\sum_{l\geq1}l^{j-1}\frac{q^l}{1+q^{2l}}\cdot
\end{eqnarray*}
This shows \eqref{Ysbetaqpair} via the definition \eqref{betaq} of $\beta_q$.
\end{proof}

\subsection{Linear combination in the $\beta_q(2j)$, $j\in\N^*$}

Define:
\begin{multline}\label{Rn}
R_n(T;q):=T^{(A-2r)n/2+A/2-2}q^{-A(n^2-1)/2-((A-2r)n+A-2)/4}\\
\times\frac{(q)_n^{A-2r}(q^{-rn}T,q^{n}T)_{rn}}{(T-q^{1/2})^{A}(T-q^{1/2-1})^{A}\dots (T-q^{1/2-n})^{A}}\cdot
\end{multline}
Then obviously we have:
$$S_n(q)=\sum_{k\geq 1}(-1)^{k+1}q^{k}(1-q^{2k+n-1})R_n(q^k;q).$$
Notice that when $n$ is odd, $R_n(T;q)$ is a rational fraction in the variable $T$ with degree $-n(A-2r)/2-A/2-2$, which is less or equal to $-3$, as $A>2r\geq2$. Assume from now on that $n$ is a fixed odd positive integer. The partial fraction expansion of $R_n(T;q)$ can be written:
$$
R_n(T;q)=\sum_{s=1}^{A}\sum_{j=0}^n\frac{c_{s,j,n}(q)}{(T-q^{1/2-j})^s}=\sum_{s=1}^{A}\sum_{j=0}^n\frac{d_{s,j,n}(q)}{(1-Tq^{j-1/2})^s},
$$
where 
\begin{equation}\label{Alpha}
d_{s,j,n}(q):=(-1)^sq^{(j-1/2)s}c_{s,j,n}(q)
\end{equation}
and
\begin{eqnarray}
c_{s,j,n}(q)&:=&\frac{1}{(A-s)!}\left[\frac{d^{A-s}}{dT^{A-s}}R_n(T;q)(T-q^{1/2-j})^{A}\right]_{T=q^{1/2-j}}\label{c}\\
&=&\frac{q^{(1/2-j)s}}{(A-s)!}\left[\frac{d^{A-s}}{du^{A-s}}R_n(uq^{1/2-j};q)(u-1)^{A}\right]_{u=1}.\label{c2}
\end{eqnarray}

\medskip
\noindent The definition \eqref{Rn} gives
$R_n(Tq^{n-1};1/q)=-q^{n(r-2)+1}R_n(T;q)$, which yields for all $j\in\{0,\dots,n\}$ and $s\in\{1,\dots,A\}$:
\begin{equation}\label{relationdj}
d_{s,n-j,n}(1/q)=-q^{n(r-2)+1}d_{s,j,n}(q),
\end{equation}
or equivalently
\begin{equation}\label{relationcj}
c_{s,n-j,n}(1/q)=-q^{n(s+r-2)+1-s}c_{s,j,n}(q).
\end{equation}

We can now prove the following Lemma, which gives explicitely the expected linear combination in terms of the values of $\beta_q$ at even positive integers:
\begin{lem}
For all $0<|q|<1$, any odd positive integers $n$ and $A$, and $r\in\N^*$ such that $A>2r$:
\begin{equation}\label{formelineaire}
S_n(q^2)=\hat{P}_{0,n}(q^2)+\sum_{{j=2\atop
j\,\mbox{\scriptsize {even}}}}^{A-1}\hat{P}_{j,n}(q^2)\beta_q(j),
\end{equation}
where for $j=2,4,\dots,A-1$,
\begin{eqnarray}
\hat{P}_{0,n}(q)&:=&P_{0,n}(1,q)+q^{-n(r-1)}P_{0,n}(1,1/q)-\frac{1}{2}P_{1,n}(1,q)\label{P0chap},\\
\hat{P}_{j,n}(q)&:=&\sum_{s=j}^{A}\frac{2c(s-1,j-1)}{(s-1)!}P_{s,n}(1,q)\label{Pjchap},
\end{eqnarray}
and
\begin{eqnarray}
P_{0,n}(z,q)&:=&\sum_{s=1}^{A}\sum_{j=1}^n\sum_{k=1}^j(-1)^{j+k}\frac{q^{k-j}}{(1-q^{k-1/2})^s}\,d_{s,j,n}(q)z^{j-k}\label{P0},\\
P_{s,n}(z,q)&:=&\sum_{j=0}^n(-1)^jq^{1/2-j}d_{s,j,n}(q)z^j\label{Pj}.
\end{eqnarray}
\end{lem}
\begin{proof}
We fix an odd positive integer $n$ and a complex number $q$ such that $0<|q|<1$. Define for complex $z$
$$
{\cal{S}}_n(z;q):=\sum_{k\geq 1}(-1)^{k+1}q^kR_n(q^k;q)z^{-k}.
$$
It is not difficult to see that ${\cal{S}}_n(z;q)$ converges
if $|z|>|q|^{(A-2r)n/2+A/2-1}$ and ${\cal{S}}_n(1/z;1/q)$
converges if $|z|<|q|^{-(A-2r)n/2-A/2}$. Thus both series ${\cal{S}}_n(z;q)$ and ${\cal{S}}_n(1/z;1/q)$ converge if $|q|<|z|\leq1$. A direct calculation shows that:
\begin{equation}\label{**}
{\cal{S}}_n(1;q)+q^{-n(r-1)}{\cal{S}}_n(1;1/q)=S_n(q).
\end{equation}

\medskip
\noindent 
On the other hand, by using the partial fraction expansion of $R_n(T;q)$, we get:
\begin{equation}\label{devcalSn}
{\cal{S}}_n(z;q)=P_{0,n}(z,q)+\sum_{s=1}^{A}P_{s,n}(z,q)L_s(z;q),
\end{equation}
where $P_{0,n}(z,q)$ and $P_{s,n}(z,q)$ are polynomials in the variable
$z$ defined by \eqref{P0} and \eqref{Pj}, and
\begin{equation}\label{defLs}
L_s(z;q):=\sum_{k\geq1}(-1)^{k+1}\frac{q^{k-1/2}}{(1-q^{k-1/2})^s}z^{-k}.
\end{equation}
Coming back to the definition \eqref{Pj}, we see that \eqref{relationdj} implies that for $s\geq1$, $P_{s,n}(1/z,1/q)=z^{-n}q^{n(r-1)}P_{s,n}(z,q)$. This gives us the idea to study the series:
\begin{equation}\label{defcalSntilde}
\tilde{{\cal{S}}}_n(z;q):={\cal{S}}_n(z;q)+z^{n}q^{-n(r-1)}{\cal{S}}_n(1/z;1/q),
\end{equation}
with the convergence condition $|q|<|z|<1$. The expansion \eqref{devcalSn}, and the previous relation between
$P_{s,n}(1/z,1/q)$ and $P_{s,n}(z,q)$, give for $|q|<|z|<1$:
\begin{multline}\label{devcalSntilde}
\tilde{{\cal{S}}}_n(z;q)=P_{0,n}(z,q)+q^{-n(r-1)}z^{n}P_{0,n}(1/z,1/q)\\
+\sum_{s=1}^{A}P_{s,n}(z,q)(L_s(z;q)+L_s(1/z;1/q)).
\end{multline}
Note that we can deduce directly from the definition \eqref{defLs}:
$$L_1(1/z;1/q)=-\sum_{k\geq1}\sum_{l\geq0}(-1)^{k+1}q^{(k-1/2)l}z^{k},$$
which yields $L_1(1/z;1/q)=-L_1(1/z;q)-z/(1+z)$, and therefore:  
\begin{equation}\label{l1}
\lim_{z\to1}P_{1,n}(z,q)(L_1(z;q)+L_1(1/z;1/q))=-\frac{1}{2}P_{1,n}(1,q).
\end{equation}
It remains to let $z$ tend to 1 in
\eqref{devcalSntilde} to get the Lemma, via the definition \eqref{defcalSntilde}, the relations \eqref{**}, \eqref{Ysbetaqpair} and \eqref{l1}, together with the fact that $L_s(1;q^2)=Y_s(q)$ for $s\geq2$. 
\end{proof}

\section{Arithmetical investigations}

In this section, we find for any odd positive integer $n$ a common denominator $D_n(q)$ to the coefficients $\hat{P}_{j,n}(q^2)\in\Q(q)$. Let 
$d_n(x)\in\Z[x]$ be the unitary polynomial, with lowest degree, and common multiple of  $1-x,1-x^2,\dots,1-x^n$. Recall some standard properties on cyclotomic polynomials. For $t\in\N$, the $t$-th cyclotomic polynomial is defined by $\phi_t(x):=\prod_{k\wedge t=1,k\leq t}(x-\mbox{e}^{2ik\pi/t})$,
and satisfies $\phi_t(x)\in\Z[x]$. Then one can prove:
\begin{equation}\label{1000}
x^n-1=\prod_{d|n}\phi_d(x),
\end{equation}
 which yields:
\begin{equation}\label{dnaveccyclot}
d_n(x)=\prod_{t=1}^n\phi_t(x).
\end{equation}
We will also need the following polynomials:
\begin{equation}\label{deltanaveccyclot}
\Delta_n(x):=\prod_{t=1\atop t\,\mbox{\scriptsize{odd}}}^{2n-1}\phi_t(x),
\end{equation}
and
\begin{equation}\label{phiprod}
\varphi_n(x):=\phi_2(x)^n\phi_4(x)^{\left\lfloor n/2\right\rfloor}\dots\phi_{2n}(x),
\end{equation}
where $\left\lfloor x\right\rfloor$ is the integer part of the real number $x$. In what follows, we denote by $ord_{\phi_t(x)}(Q(x))$ the greatest power of $\phi_t(x)$ dividing the polynomial $Q(x)$. We have the following useful arithmetical Lemma:
\begin{lem}
Let $e$ be an odd positive integer and $w_n(x):=\prod_{i=1}^n\frac{1-x^{e+2i}}{1-x^{2i}}$. Then
$$\varphi_n(x)\,w_n(x)\in\mathbb{Z}\left[x\right].$$
\end{lem}
\begin{rem}
\emph{We can see that $\varphi_n(x)$ always divides $\prod_{i=1}^n(1+x^i)^2$, so we could use this alternative polynomial later, instead of $\varphi_n(x)$. This would give simpler manipulations, but we would have less good asymptotic properties (see Lemma~5.4).}
\end{rem}

\begin{proof}[Proof of Lemma~4.1]
From \eqref{1000}, $w_n(x)$ is a quotient of products of cyclotomic polynomials $\phi_t(x)$. It is enough to prove that $ord_{\phi_t(x)}(\varphi_n(x)\,w_n(x))\geq0$ for all $t$. Assume first that $t$ is even. As $e$ is odd, $\phi_t(x)$ can only be a factor in the denominator of $w_n(x)$. Then we see directly that  $ord_{\phi_t(x)}(w_n(x))=-\left\lfloor \frac{2n}{t}\right\rfloor=-ord_{\phi_t(x)}(\varphi_n(x))$, thus $ord_{\phi_t(x)}(\varphi_n(x)\,w_n(x))=0$.

\medskip
\noindent Now if $t$ is odd, we have:
\begin{eqnarray*}
ord_{\phi_t(x)}(\varphi_n(x)\,w_n(x))&=&ord_{\phi_t(x)}(w_n(x))\\
&=&\sum_{j=0}^{\left\lfloor \frac{n}{t}\right\rfloor-1}ord_{\phi_t(x)}\left(\prod_{i=jt+1}^{(j+1)t}\frac{1-x^{e+2i}}{1-x^{2i}}\right)\\
&&\hskip 3cm+ord_{\phi_t(x)}\left(\prod_{i=t\left\lfloor \frac{n}{t}\right\rfloor+1}^{n}\frac{1-x^{e+2i}}{1-x^{2i}}\right).
\end{eqnarray*}
But the orders of divisibility in the sum over $j$ are all equal to $1-1=0$. Moreover we have:
$$ord_{\phi_t(x)}\left(\prod_{i=t\left\lfloor \frac{n}{t}\right\rfloor+1}^{n}\frac{1-x^{e+2i}}{1-x^{2i}}\right)=ord_{\phi_t(x)}\left(\prod_{i=t\left\lfloor \frac{n}{t}\right\rfloor+1}^{n}(1-x^{e+2i})\right)\in\{0;1\},$$
which proves the Lemma.
\end{proof}

Throughout this section, we assume  $n$ and $A$ to be odd positive integers, and $r$ to be a positive integer such that $A-2r>0$.
\begin{lem}
For all $s\in\{1,\dots,A\}$ and $j\in\{0,\dots,n\}$, we have:
$$\varphi_n(1/q)^{2r}\,d_n\left(1/q^2\right)^{A-s}c_{s,j,n}(q^2)\in\mathbb{Z}\left[q;\frac{1}{q}\right].$$
\end{lem}
\begin{proof}
Rewrite (\ref{c}) as:
\begin{equation}\label{crecrit}
c_{s,j,n}(q)=\frac{1}{(A-s)!}\left[\frac{d^{A-s}}{dT^{A-s}}V_n(T;q)\right]_{T=q^{1/2-j}},
\end{equation}
with
\begin{eqnarray}
V_n(T;q)&:=&R_n(T;q)(T-q^{1/2-j})^A\nonumber\\
&=&(q)_n^{A-2r}T^{(A-2r)n/2+A/2-2}q^{-A(n^2-1)/2-((A-2r)n+A-2)/4}\nonumber\\
&&\hskip
1cm\times\frac{(q^{-rn}T,q^{n}T)_{rn}(T-q^{1/2-j})^A}{(T-q^{1/2})^{A}(T-q^{1/2-1})^{A}\dots (T-q^{1/2-n})^{A}}\cdot\label{V}
\end{eqnarray}
We collect the terms of $V_n(T;q)$ as follows:
$$V_n(T;q)=q^{an^2+bn+c}T^{A/2-2-n/2}F(T)^{A/2-r+1/2}G(T)^{A/2-r-1/2}\prod_{l=1}^rH_l(T)I_l(T),
$$
where $a$, $b$ and $c$ are integers (or half-integers) depending only on
$A$ and $r$, and the functions $F$, $G$, $H_l$ and $I_l$ satisfy:
\begin{multline}\label{fcF}
F(T):=q^{-n(n+1)/2}\frac{(q)_nT^n(T-q^{1/2-j})}{(T-q^{1/2})(T-q^{1/2-1})\dots (T-q^{1/2-n})}\\
=(-1)^n(1/q;1/q)_n+\sum_{{i=0\atop i\neq j}}^n(-1)^{n-i+1}q^{-i(i+1)/2}\qbi{n}{i}{1/q}\frac{q^{1/2-j}-q^{1/2-i}}{T-q^{1/2-i}},
\end{multline}
\begin{multline}\label{fcG}
G(T):=q^{-n(n+1)/2}\frac{(q)_n(T-q^{1/2-j})}{(T-q^{1/2})(T-q^{1/2-1})\dots (T-q^{1/2-n})}\\
=\sum_{{i=0\atop i\neq j}}^n(-1)^{n-i+1}q^{(i-1/2)n-i(i+1)/2}\qbi{n}{i}{1/q}\frac{q^{1/2-j}-q^{1/2-i}}{T-q^{1/2-i}},
\end{multline}
\begin{multline}\label{fcHl}
H_l(T):=q^{-n(n+1)/2}\frac{(q^{-ln}T)_n(T-q^{1/2-j})}{(T-q^{1/2})(T-q^{1/2-1})\dots (T-q^{1/2-n})}\\
=(-1)^nq^{-ln^2-n}+\sum_{{i=0\atop i\neq j}}^n(-1)^{i+1}q^{-(n-i)^2/2-(2n+i)/2}\qbi{n}{i}{1/q}\\
\times\frac{(q^{-(l-1)n-i-1/2};q^{-1})_n}{(q^{-1};q^{-1})_n}\frac{q^{1/2-j}-q^{1/2-i}}{T-q^{1/2-i}},
\end{multline}
\begin{multline}\label{fcIl}
I_l(T):=q^{-n(n+1)/2}\frac{(q^{ln}T)_n(T-q^{1/2-j})}{(T-q^{1/2})(T-q^{1/2-1})\dots (T-q^{1/2-n})}\\
=(-1)^nq^{ln^2-n}+\sum_{{i=0\atop i\neq j}}^n(-1)^{n+i+1}q^{ln^2-n-i(i+1)/2}\qbi{n}{i}{1/q}\\
\times\frac{(q^{-(l+1)n-i-1/2};q^{-1})_n}{(q^{-1};q^{-1})_n}\frac{q^{1/2-j}-q^{1/2-i}}{T-q^{1/2-i}}\cdot
\end{multline}
We see that if $q$ is replaced by $q^2$ and if $U$ denotes any of the functions $F$, $G$, or $T\mapsto T^{A/2-2-n/2}$, then by using the partial fraction expansions \eqref{fcF} and \eqref{fcG}:
$$\frac{d_n\left(1/q^2\right)^{\mu}}{\mu!}\left[\frac{d^{\mu}}{dT^{\mu}}U(T)\right]_{T=q^{1-2j}}\in\mathbb{Z}\left[q;\frac{1}{q}\right]\;\;\;\forall\mu\in\N.$$ 
Now if $q$ is replaced by $q^2$ and if $U$ denotes any of the functions $H_l$ or $I_l$, then by using the partial fraction expansions \eqref{fcHl}, \eqref{fcIl}, and Lemma~4.1, we get for all $\mu\in\N$:
$$\varphi_n(1/q)\,\frac{d_n\left(1/q^2\right)^{\mu}}{\mu!}\left[\frac{d^{\mu}}{dT^{\mu}}U(T)\right]_{T=q^{1-2j}}\in\mathbb{Z}\left[q;\frac{1}{q}\right].$$ 
We can easily conclude by using (\ref{crecrit}) and by applying Leibniz's formula for the $(\mu=A-s)$-th differentiation of a product of functions.
\end{proof}

\medskip
\noindent 

\begin{lem}
Set $\alpha=-A/4-r^2$. Then there exist real numbers $\beta'$ and $\gamma'$
depending only on $A$ and $r$ such that for all $(s,j)\in\{1,\dots,A\}\times\{0,\dots,n\}$:
$$\lim_{q\to+\infty}q^{\alpha n^2+\beta'
n+\gamma'}c_{s,j,n}(q^2)<\infty.$$
\end{lem}
\begin{proof}
 Recall the expression (\ref{crecrit}) from the previous proof,
and let
$v_n(T;q):=\frac{d}{dT}V_n(T;q)/V_n(T;q)$
be the logarithmic derivative of $V_n(T;q)$, in the variable $T$. As in \cite{KRZ} and \cite{JM}, we use Fa\`a di Bruno's differentiation formula, which gives for all $\mu\in\N$:
\begin{equation}\label{FaaVn}
\frac{1}{\mu!}\frac{d^{\mu}}{dT^{\mu}}V_n(T;q)=\sum_{k_1+\dots+\mu
k_\mu=\mu}\frac{V_n(T;q)}{k_1!\dots
k_\mu!}\prod_{l=1}^\mu\left(\frac{1}{l!}\frac{d^{l-1}}{dT^{l-1}}v_n(T;q)\right)^{k_l}.
\end{equation}
By \eqref{V} we get:
\begin{eqnarray*}
v_n(T;q)&=&\frac{d}{dT}(\log V_n(T;q))\\
&=&\frac{(A-2r)n/2+A/2-2}{T}+\sum_{i=1}^{rn}\frac{1}{T-q^i}\\
&&\hskip 3cm
+\sum_{i=n}^{rn+n-1}\frac{1}{T-q^{-i}}-A\sum_{{i=0\atop i\neq
j}}^{n}\frac{1}{T-q^{1/2-i}},
\end{eqnarray*}
thus for all $l\in\N^*$:
\begin{multline*}
\frac{(-1)^{l-1}}{(l-1)!}\frac{d^{l-1}}{dT^{l-1}}v_n(T;q)=\frac{(A-2r)n/2+A/2-2}{T^l}+\sum_{i=1}^{rn}\frac{1}{(T-q^i)^l}\\
+\sum_{i=n}^{rn+n-1}\frac{1}{(T-q^{-i})^l}-A\sum_{{i=0\atop i\neq
j}}^{n}\frac{1}{(T-q^{1/2-i})^l}\cdot
\end{multline*}
We can rewrite this as follows:
\begin{multline*}
\frac{(-1)^{l-1}}{(l-1)!}\frac{d^{l-1}}{dT^{l-1}}v_n(T;q)=\frac{(A-2r)n/2+A/2-2}{T^l}+\sum_{i=1}^{rn}\left(\frac{q^{-i}}{Tq^{-i}-1}\right)^l\\
+\sum_{i=n}^{rn+n-1}\frac{1}{(T-q^{-i})^l}-A\sum_{i=0}^{j-1}\left(\frac{q^{i}}{Tq^{i}-1}\right)^l-A\sum_{i=j+1}^{n}\frac{1}{(T-q^{-i})^l}\cdot
\end{multline*}
Therefore $\forall j\in\{0,\dots,n\}$,
$\displaystyle\lim_{q\to+\infty}q^{(1/2-j)l}\left[\frac{d^{l-1}}{dT^{l-1}}v_n(T;q)\right]_{T=q^{1/2-j}}<\infty$.
We deduce that for $k_1+\dots+\mu k_\mu=\mu$:
\begin{equation}\label{E}
\lim_{q\to+\infty}q^{(1/2-j)\mu}\left[\prod_{l=1}^\mu\left(\frac{1}{l!}\frac{d^{l-1}}{dT^{l-1}}v_n(T;q)\right)^{k_l}\right]_{T=q^{1/2-j}}<\infty.
\end{equation}
Besides, $V_n(q^{1/2-j};q)$ defined by \eqref{V} satisfies:
$$\lim_{q\to+\infty}V_n(q^{1/2-j};q)\times q^{-j(An-A+3)/2+j^2A/2-rn(rn-2)/2-A/2}=1.$$

\medskip
\noindent 
It remains to choose $\mu=A-s$ in (\ref{FaaVn}). With the help of (\ref{crecrit}) and  \eqref{E} we then get:
$$\lim_{q\to+\infty}q^{(1/2-j)(A-s)-\left(j(An-A+3)/2-j^2A/2+rn(rn-2)/2+A/2\right)}c_{s,j,n}(q)<\infty\;.$$
Replacing $q$ by $q^2$, we can easily conclude, for we have: $\forall (s,j)\in\{1,\dots,A\}\times\{0,\dots,n\}$,
$$(1-2j)(A-s)-\left(j(An-A+3)-j^2A+rn(rn-2)+A\right)\geq \alpha n^2+\beta' n+\gamma'\;,$$
where $\alpha=-A/4-r^2$, $\beta'=2r-(A+1)/2$, $\gamma'=-1/2-1/(2A)$ (this lower bound is obtained for $s=1$ and $j=(n+1)/2+1/(2A)$).
\end{proof}

\medskip
Now we can prove the following Lemma, which gives an expression for $D_n(q)$, a common denominator to $\hat{P}_{j,n}(q^2)$, for $j\in\{0,2,4,\dots,A-1\}$:
\begin{lem}
Let $n$ and $A$ be odd positive integers, and $r\in\N^*$ such that $A-2r>0$. For $\alpha=-A/4-r^2$, there exist $\beta$ and $\gamma$ real numbers depending only on $A$ and $r$ such that for all $j\in\{2,4,\dots,A-1\}$:
\begin{equation}\label{denominatpjchap}
(A-1)!\,q^{\lfloor\alpha n^2+\beta
n+\gamma\rfloor}\varphi_n(1/q)^{2r}\,d_n(1/q^2)^{A-j}\hat{P}_{j,n}(q^2)\in\Z\left[\frac{1}{q}\right]
\end{equation}
and
\begin{equation}\label{denominatp0chap}
q^{\lfloor\alpha n^2+\beta
n+\gamma\rfloor}\varphi_n(1/q)^{2r}\,d_{2n}(1/q)^{A-1}\,\Delta_n(1/q)\,\hat{P}_{0,n}(q^2)\in\Z\left[\frac{1}{q}\right].
\end{equation}
Thus, by setting
\begin{equation}\label{DenominateurcommunDn}
D_n(q):=(A-1)!\,q^{\lfloor\alpha n^2+\beta
n+\gamma\rfloor}\varphi_n(1/q)^{2r}\,d_{2n}(1/q)^{A-1}\,\Delta_n(1/q),
\end{equation}
we get:
$$D_n(q)\hat{P}_{j,n}(q^2)\in\Z\left[\frac{1}{q}\right]\;\;\forall
j\in\{0,2,4,\dots,A-1\}.$$
\end{lem}
\begin{proof}
As $\lim_{q\to+\infty}d_n(1/q^2)=d_n(0)=\pm1$ and $\lim_{q\to+\infty}\phi_t(1/q)=\phi_t(0)=\pm1$, the previous Lemma implies that for all
$(s,j)\in\{1,\dots,A\}\times\{0,\dots,n\}$ we have:
$$\lim_{q\to+\infty}q^{\alpha n^2+\beta'
n+\gamma'}\varphi_n(1/q)^{2r}\,d_n(1/q^2)^{A-s}c_{s,j,n}(q^2)<\infty.$$
 By using Lemma~4.3 we then obtain that for $\alpha=-A/4-r^2$ there exist
real numbers $\beta'$ and $\gamma'$ depending only on $A$ and $r$, such that for all $ (s,j)\in\{1,\dots,A\}\times\{0,\dots,n\}$:
\begin{equation}\label{100}
q^{\lfloor\alpha n^2+\beta'
n+\gamma'\rfloor}\varphi_n(1/q)^{2r}\,d_n\left(1/q^2\right)^{A-s}c_{s,j,n}(q^2)\in\Z\left[\frac{1}{q}\right].
\end{equation}
From the expressions \eqref{Alpha}, (\ref{Pjchap}) and (\ref{Pj}), we deduce \eqref{denominatpjchap} (some easy computations show that the values $\beta=\beta'-A+1$ and $\gamma=\gamma'+A-2$ are convenient).

\medskip
Besides, the definition \eqref{deltanaveccyclot} shows that $\Delta_n(x)$ is nothing else but the lowest common multiple (lcm) of $1-x$, $1-x^3$,\dots,$1-x^{2n-1}$. Recall that $d_n(x)$ is the lcm of $1-x$, $1-x^2$,\dots,$1-x^n$. From the definition of lcm and equation \eqref{1000}, we deduce first
$$d_n(x^2)=\prod_{t=1}^n\phi_t(x)\times\prod_{t=n+1\atop t\,\mbox{\scriptsize{even}}}^{2n}\phi_t(x),$$
and then that $d_{2n}(x)^{A-1}\Delta_n(x)$ is the lcm of the polynomials $\Delta_n(x)^sd_n(x^2)^{A-s}$ when $s$ runs along $\{1,\dots,A\}$. This yields \eqref{denominatp0chap}, with the help of \eqref{100}, expressions (\ref{P0chap}), (\ref{P0}) and (\ref{Pj}), and equation \eqref{relationcj}.
\end{proof}

\section{Asymptotic estimates}

We now evaluate the asymptotics for $S_n(q)$, the
coefficients $\hat{P}_{j,n}(q)$ from (\ref{formelineaire}), and finally $D_n(q)$. Throughout this section, we fix an odd integer $A$ and $r\in\N^*$ such that $A-2r>0$.

\subsection{Asymptotic evaluation of $S_n(q)$}

\begin{lem}
For all $0<|q|<1$, we have:
$$\lim_{n\to+\infty}\frac{1}{n^2}\log|S_n(q)|=-\frac{1}{2}r(A-2r)\log|1/q|.$$
\end{lem}
\begin{proof}
Set $\rho_k(q):=q^k(1-q^{2k+n-1})R_n(q^k;q)$, so that
$S_n(q)=\sum_{k\geq 1}\rho_k(q)$. By the definition \eqref{Rn}
 of $R_n(T;q)$, it is clear that
$\rho_k(q)=0\Leftrightarrow k\in\{0,\dots,rn\}$. Moreover we have for $k\geq rn+1$:
\begin{multline*}
\frac{\rho_{k+1}(q)}{\rho_k(q)}=q^{(A-2r)n/2+A/2-1}\frac{1-q^{2k+n+1}}{1-q^{2k+n-1}}\frac{1-q^{k}}{1-q^{k-rn}}\\
\times\frac{1-q^{k+n+rn}}{1-q^{k+n}}\left(\frac{1-q^{k-1/2}}{1-q^{k+n-1/2}}\right)^{A+1}.
\end{multline*}
Then, as $A-2r>0$, $0<|q|<1$ and $k\geq rn+1$, we have for a sufficiently large $n$ the following upper bound,  uniformly in $k$:
$$\left|\frac{\rho_{k+1}(q)}{\rho_k(q)}\right|\leq|q|^{(A-2r)n/2}\left(\frac{1+|q|}{1-|q|}\right)^{A+3}<\frac{1}{3},$$
which yields, as in \cite{KRZ} and \cite{JM}, the following inequalities:
$$\frac{1}{2}|\rho_{rn+1}(q)|\leq|S_n(q)|\leq\frac{3}{2}|\rho_{rn+1}(q)|.$$
Besides,
\begin{multline*}
\rho_{rn+1}(q)=(-1)^{rn+2}q^{(rn+1/2)((A-2r)n/2+A/2-1)}\\
\times(1-q^{2rn+n+1})(q)_n^{A-2r}\frac{(q,q^{(r+1)n+1})_{rn}}{(q^{rn+1/2})_{n+1}^{A}},
\end{multline*}
therefore we get:
$$\lim_{n\to+\infty}\frac{1}{n^2}\log|S_n(q)|=\lim_{n\to+\infty}\frac{1}{n^2}\log|\rho_{rn+1}(q)|=-\frac{1}{2}r(A-2r)\log|1/q|.$$
\end{proof}

\subsection{Asymptotic evaluation of the coefficients $\hat{P}_{j,n}(q)$ of (\ref{formelineaire})}

\begin{lem}
For all $j\in\{0,2,4,\dots,A-1\}$ and $0<|q|<1$, we have:
$$\limsup_{{n\to+\infty\atop n\,\mbox{\scriptsize{odd}}}}\frac{1}{n^2}\log|\hat{P}_{j,n}(q)|\leq
\frac{1}{8}(A+4r^2)\log|1/q|.$$
\end{lem}
\begin{proof}
We assume from now on that $n$ is an odd positive integer. First note that for any complex numbers $a_{i,n}$
($0\leq i\leq n$):
$$\left(\forall
i\in\{0,\dots,n\},\,\limsup_{n\to+\infty}\frac{1}{n^2}\log|a_{i,n}|\leq
c\right)\,\Rightarrow\limsup_{n\to+\infty}\frac{1}{n^2}\log\left|\sum_{i=0}^na_{i,n}\right|\leq
c.$$
This shows, via the definition of the coefficients $\hat{P}_{j,n}(q)$ given by
(\ref{P0chap})-(\ref{Pj}), that it is enough to prove the inequality of the Lemma for the coefficients $d_{s,j,n}(q)=(-1)^sq^{(j-1/2)s}c_{s,j,n}(q)$,
uniformly in $j$ and $s$. We fix the integer $j\in\{0,\dots,n\}$ and
$\eta=(1-|q|)/2>0$. Cauchy's formula applied to (\ref{c}) gives:
$$d_{s,j,n}(q)=-\frac{1}{2i\pi}\int_{\cal
C}R_n(Tq^{1/2-j};q)(1-T)^{s-1}dT,$$ where ${\cal C}$ is the circle of center 1 and radius $\eta$. Back to the expression (\ref{Rn}), we get:
\begin{multline*}
R_n(Tq^{1/2-j};q)(1-T)^{s-1}=q^{-j((A-2r)n/2+A/2-2)-An-1/2}T^{(A-2r)n/2+A/2-2}\\
\times(q)_n^{A-2r}(1-T)^{s-1}\frac{(q^{-rn-j+1/2}T,q^{n-j+1/2}T)_{rn}}{(Tq^{-j})_{n+1}^A}\cdot
\end{multline*}
After some elementary manipulations, we can deduce:
\begin{multline*}
R_n(Tq^{1/2-j};q)(1-T)^{s-1}=q^{Aj^2/2-Anj/2+2j-(rn)^2/2-An-1/2}T^{A(n-2j)/2+A/2-2}\\
\times(-1)^{Aj+rn}(q)_n^{A-2r}(1-T)^{s-A-1}\frac{(q^{j+1/2}/T,q^{n-j+1/2}T)_{rn}}{(q/T)_{j}^A(qT)_{n-j}^A}\cdot
\end{multline*}
In order to find an upper bound to this expression for $T\in{\cal C}$, we use the following inequalities from \cite{KRZ}, valid for
$(a,b)\in\N^*\times\N$,  $T\in{\cal C}$ and $\eta=(1-|q|)/2$:
$$0<(|q|(1+\eta);|q|)_\infty\leq|(q^aT)_b|\leq(-(1+\eta);|q|)_\infty,$$
$$0<(|q|/(1-\eta);|q|)_\infty\leq|(q^a/T)_b|\leq(-1/(1-\eta);|q|)_\infty,$$
$$|T^{A(n-2j)/2+A/2-2}|\leq (\max(1+\eta;1/(1-\eta))^{An/2}(1+\eta)^{A/2-2},$$
$$|(q)_n|\leq(-|q|;|q|)_\infty\;\;\mbox{and}\;\;|1-T|^{s-A-1}\leq1/\eta^{A+1}.$$

\medskip
\noindent 
It remains to find a lower bound for the power of $q$ in the previous expression of $R_n(Tq^{1/2-j};q)(1-T)^{s-1}$. This can be done by noting that the function $j\mapsto
Aj^2/2-Anj/2+2j-(rn)^2/2-An-1/2$ is minimal at $j=n/2-2/A$, and this minimal value is  equal to $-An^2/8-r^2n^2/2+\la
n+\mu$, where $\lambda$ and $\mu$ are real numbers depending only on $A$ and $r$. All this yields to the following:
$$|d_{s,j,n}(q)|\leq c_0\times|q|^{-(A+4r^2)n^2/8},$$
where $c_0$ does not depend on $j$ neither $s$, and satisfies
$\displaystyle\lim_{n\to+\infty}c_0^{\;1/n^2}=1$, and we can conclude.
\end{proof}

\subsection{Asymptotic evaluation of $D_n(q)$ defined by (\ref{DenominateurcommunDn})}

We first prove a preliminary result:
\begin{lem} For any positive integer $n$, we have:
\begin{eqnarray}
\sum_{1\leq d\leq n\atop d\,\mbox{\scriptsize{odd}}}\frac{\mu(d)}{d^2}&=&\frac{8}{\pi^2}+\mbox{O}(1/n)\label{muimpair},\\
\sum_{1\leq d\leq n\atop d\,\mbox{\scriptsize{even}}}\frac{\mu(d)}{d^2}&=&-\frac{2}{\pi^2}+\mbox{O}(1/n)\label{mupair},
\end{eqnarray}
where $\mu$ is the M\"obius function.
\end{lem}
\begin{proof}
Recall
$$
\frac{6}{\pi^2}=\sum_{d\geq1}\frac{\mu(d)}{d^2}=\sum_{d\geq1}\frac{\mu(2d)}{4d^2}+\sum_{d\geq1\atop d\,\mbox{\scriptsize{odd}}}\frac{\mu(d)}{d^2}.
$$
Besides $\mu(2d)=-\mu(d)$ if $d$ is odd and $\mu(2d)=0$ if $d$ is even, so we get:
\begin{equation}\label{muimp}
\sum_{d\geq1\atop d\,\mbox{\scriptsize{odd}}}\frac{\mu(d)}{d^2}=\frac{8}{\pi^2},
\end{equation} 
and this implies
\begin{equation}\label{mup}
\sum_{d\geq1\atop d\,\mbox{\scriptsize{even}}}\frac{\mu(d)}{d^2}=-\frac{2}{\pi^2}\cdot
\end{equation} 
Then we immediately deduce \eqref{muimpair} (resp. \eqref{mupair}) from \eqref{muimp} (resp. \eqref{mup}). 
\end{proof}
Now we can prove the following Lemma:
\begin{lem}
For all $0<|q|<1$ we have:
\begin{eqnarray}
\lim_{n\to+\infty}\frac{1}{n^2}\log|\Delta_n(1/q)|&=&\frac{8}{\pi^2}\log|1/q|\label{asyptdelta},\\
\lim_{n\to+\infty}\frac{1}{n^2}\log|\varphi_n(1/q)|&=&\frac{2}{3}\log|1/q|\label{asymptphi}.
\end{eqnarray}
\end{lem}
\begin{proof}
M\"obius inversion formula applied to \eqref{1000} yields for $k\in\N^*$:
\begin{equation}\label{moebius}
\log|\phi_k(x)|=\sum_{d|k}\mu(d)\log|x^{k/d}-1|.
\end{equation}

First recall the definition \eqref{deltanaveccyclot}:
$$\Delta_n(x)=\prod_{k=1\atop k\,\mbox{\scriptsize{odd}}}^{2n-1}\phi_k(x).$$
Thus, by using \eqref{moebius}, we can write for $|x|>1$:
\begin{eqnarray*}
\log|\Delta_n(x)|&=&\sum_{k=1\atop k\,\mbox{\scriptsize{odd}}}^{2n}\sum_{d|k}\mu(d)\log|x^{k/d}-1|=\sum_{1\leq d\leq2n\atop d\,\mbox{\scriptsize{odd}}}\mu(d)\sum_{1\leq l\leq2n/d\atop l\,\mbox{\scriptsize{odd}}}\log|x^{l}-1|\\
&=&\sum_{1\leq d\leq2n\atop d\,\mbox{\scriptsize{odd}}}\mu(d)\sum_{1\leq l\leq2n/d\atop l\,\mbox{\scriptsize{odd}}}(l\log|x|+\log|1-x^{-l}|)\\
&=&\log|x|\sum_{1\leq d\leq2n\atop d\,\mbox{\scriptsize{odd}}}\mu(d)\sum_{1\leq l\leq2n/d\atop l\,\mbox{\scriptsize{odd}}}l+\mbox{O}\left(n\sum_{l=1}^{2n}\frac{1}{l}\right).
\end{eqnarray*}
Since $\displaystyle\sum_{1\leq l\leq2n/d\atop l\,\mbox{\scriptsize{odd}}}l=\frac{n^2}{d^2}+\mbox{O}(1)$ and $\displaystyle\mbox{O}\left(n\sum_{l=1}^{2n}\frac{1}{l}\right)=\mbox{O}(n\log n)$, we get:
$$\log|\Delta_n(x)|=n^2\log|x|\sum_{1\leq d\leq2n\atop d\,\mbox{\scriptsize{odd}}}\frac{\mu(d)}{d^2}+\mbox{O}(n\log n),$$
which yields \eqref{asyptdelta}, by setting $x=1/q$ and using \eqref{muimpair}.\\

\medskip
\noindent To prove \eqref{asymptphi}, recall the definition \eqref{phiprod}:
$$\varphi_n(x)=\phi_2(x)^n\phi_4(x)^{\left\lfloor n/2\right\rfloor}\dots\phi_{2n}(x).$$
We use again \eqref{moebius}, assuming $|x|>1$:
\begin{eqnarray*}
\log|\varphi_n(x)|&=&\sum_{k=1}^{n}\left\lfloor \frac{n}{k}\right\rfloor\sum_{d|2k}\mu(d)\log|x^{2k/d}-1|\\
&=&\sum_{d,l\geq1\atop dl\,\mbox{\scriptsize{even}}\leq2n}\mu(d)\left\lfloor \frac{2n}{dl}\right\rfloor\log|x^{l}-1|\\
&=&\sum_{1\leq d\leq2n\atop d\,\mbox{\scriptsize{even}}}\mu(d)\sum_{1\leq l\leq2n/d}\left\lfloor \frac{2n}{dl}\right\rfloor\log|x^{l}-1|\\
&&\hskip3cm+\sum_{1\leq d\leq2n\atop d\,\mbox{\scriptsize{odd}}}\mu(d)\sum_{1\leq l\leq2n/d\atop l\,\mbox{\scriptsize{even}}}\left\lfloor \frac{2n}{dl}\right\rfloor\log|x^{l}-1|.
\end{eqnarray*}
Hence, by setting $u_n(x):=\sum_{l=1}^n\left\lfloor \frac{n}{l}\right\rfloor\log|x^{l}-1|$, we have:
\begin{equation}\label{exprphin}
\log|\varphi_n(x)|=\sum_{1\leq d\leq2n\atop d\,\mbox{\scriptsize{even}}}\mu(d)u_{2n/d}(x)+\sum_{1\leq d\leq2n\atop d\,\mbox{\scriptsize{odd}}}\mu(d)u_{n/d}(x^2).
\end{equation}
Assuming $|x|>1$ we can write:
\begin{eqnarray*}
u_n(x)&=&\log|x|\sum_{l=1}^nl\left\lfloor \frac{n}{l}\right\rfloor+\sum_{l=1}^n\left\lfloor \frac{n}{l}\right\rfloor\log|1-x^{-l}|\\
&=&\log|x|\sum_{l=1}^nl\sum_{1\leq k\leq n/l}1+\sum_{l,k\geq 1\atop lk\leq n}\log|1-x^{-l}|\\
&=&\frac{\log|x|}{2}\sum_{k=1}^n\left\lfloor\frac{n}{k}\right\rfloor\left(\left\lfloor\frac{n}{k}\right\rfloor+1\right)+\mbox{O}\left(\sum_{l,k\geq 1\atop lk\leq n}1\right)\\
&=&\frac{\log|x|}{2}\sum_{k=1}^n\frac{n^2}{k^2}+\mbox{O}\left(n\log n\right)\\
&=&n^2\log|x|\times\frac{\pi^2}{12}+\mbox{O}\left(n\log n\right).
\end{eqnarray*}
Therefore, by using \eqref{exprphin}, \eqref{mupair} and \eqref{muimpair}, we deduce:
\begin{eqnarray*}
\log|\varphi_n(x)|&=&\sum_{1\leq d\leq2n\atop d\,\mbox{\scriptsize{even}}}\mu(d)\left(4\frac{n^2}{d^2}\log|x|\times\frac{\pi^2}{12}+\mbox{O}\left(\frac{2n}{d}\log \frac{2n}{d}\right)\right)\\
&&\hskip2cm+\sum_{1\leq d\leq2n\atop d\,\mbox{\scriptsize{odd}}}\mu(d)\left(2\frac{n^2}{d^2}\log|x|\times\frac{\pi^2}{12}+\mbox{O}\left(\frac{n}{d}\log \frac{n}{d}\right)\right)\\
&=&n^2\log|x|\times\frac{\pi^2}{3}\times\frac{-2}{\pi^2}+n^2\log|x|\times\frac{\pi^2}{6}\times\frac{8}{\pi^2}+\mbox{O}\left(n\log^2 n\right)\\
&=&\frac{2}{3}n^2\log|x|+\mbox{O}\left(n\log^2 n\right),
\end{eqnarray*}
which, by setting $x=1/q$, shows \eqref{asymptphi} as desired.
\end{proof}
Now we are able to prove the following result:
\begin{lem}
For all $0<|q|<1$ we have:
$$\lim_{n\to+\infty}\frac{1}{n^2}\log|D_n(q)|=\left(\frac{A}{4}+r^2+\frac{12}{\pi^2}(A-1)+\frac{4r}{3}+\frac{8}{\pi^2}\right)\log|1/q|.$$
\end{lem}
\begin{proof}
For $0<|q|<1$ recall the estimate (see \cite{BV} and \cite{VA}):
\begin{equation}\label{dnasympt}
\lim_{n\to+\infty}\frac{1}{n^2}\log|d_n(1/q)|=\frac{3}{\pi^2}\log|1/q|.
\end{equation} 
Thus, with the help of the expressions of $D_n(q)$ and $\alpha=-A/4-r^2$ given by \eqref{DenominateurcommunDn}, and by using the previous Lemma, we can easily conclude.
\end{proof}

\section{Concluding remarks}

\subsection{Link with Dirichlet's beta function}

As mentioned in the introduction, we now justify the term $q$-analogues of the values of Dirichlet's beta function at positive integers. To this aim, recall the \emph{Stirling numbers of the second kind} (see \cite{St}), which are integers denoted by $S(s,j)$ (where $s$ and $j$ are integers such that $1\leq j\leq s$) and defined by:
$$x^s=\sum_{j=1}^sS(s,j)x(x-1)\dots(x-j+1).$$
By expanding the summand in the first expression of \eqref{betaq}, we can write for all $s\geq2$ and $|q|<1$:
\begin{eqnarray}
\beta_q(s)&=&\sum_{k\geq1}k^{s-1}\sum_{m\geq0}(-1)^mq^{(2m+1)k}\nonumber\\
&=&\sum_{k,m\geq0}(k+1)^{s-1}(-1)^mq^{(2m+1)(k+1)}\nonumber\\
&=&\sum_{k,m\geq0}\sum_{j=1}^{s-1}q^{(2m+1)(k+1)}(-1)^{m+s-1-j}S(s-1,j)j!\bi{k+j}{j}\nonumber\\
&=&\sum_{j=1}^{s-1}(-1)^{s-1-j}S(s-1,j)j!\sum_{m\geq0}(-1)^{m}\frac{q^{2m+1}}{(1-q^{2m+1})^{j+1}}\nonumber\\
&=&\sum_{j=1}^{s-1}(-1)^{s-1-j}S(s-1,j)j!\,Y_{j+1}(q).\label{justifqanalog}
\end{eqnarray}
Note in passing that \eqref{justifqanalog} is the inverse expansion of \eqref{Ysbetaq}. Notice that our auxiliary functions $Y_s(q)$ are clearly $q$-analogues of the function $\beta$ at positive integers, since for $s\geq1$:
$$\lim_{q\to1}(1-q)^sY_s(q)=\beta(s).$$
This shows, by using \eqref{justifqanalog} and \eqref{betaq1}, that for all $s\geq1$:
\begin{equation}\label{qanal}
\lim_{q\to1}(1-q)^s\beta_q(s)=(s-1)!\beta(s).
\end{equation}

\subsection{Special emphazis on  $\beta_q(1)$ and $\beta_q(2)$}
For $s=1$, we have in fact:
$$\beta_q(1)=Y_1(q),$$
which, as mentioned in the introduction, is, up to constants, equal to $\pi_{q}$ whose irrationality exponent was studied in \cite{BZ, BZ2}.\\
 Now we inspect more carefully the link with Catalan's constant, which is defined by $G:=\sum_{k\geq0}(-1)^{k}/(2k+1)^2=\beta(2)$. The $q$-analogue of $G$ proposed at the end of \cite{BZ2} corresponds to $Y_2(q)$. Although this is not really obvious from the definition of $\beta_q(2)$, we have in fact via \eqref{justifqanalog} the following identity, which can also be deduced from \eqref{Ysbetaq}:
$$\beta_q(2)=Y_2(q)=\sum_{k\geq0}(-1)^{k}\frac{q^{2k+1}}{(1-q^{2k+1})^{2}}\cdot$$
There are many similarities between the diophantine behaviour of the values of Riemann's zeta function at even positive integers and the values of Dirichlet's beta function at odd positive integers. However, no analogy to Apery's famous result \cite{Ap}  $\zeta(3)\notin\Q$ has been found for $G$ yet. Indeed, the linear forms built in \cite{RZ} do not show that $G$ is irrational. Moreover, even the denominators Conjecture formulated in \cite{RZ} do not give the arithmetic nature of $G$. We point out that this denominators Conjecture was proved in \cite{R} through Pad\'e approximants, and then in a simpler way by using transformation formulae for hypergeometric series in \cite{KR}. \\
Now recall the linear combination for $G$ studied in \cite{KR}:
\begin{equation}\label{GKR}
n!\sum_{k\geq
1}(-1)^{k}\left(k+\frac{n-1}{2}\right)\frac{(k-n)_n(k+n)_{n}}{(k-1/2)_{n+1}^{3}}=a_nG-b_n,
\end{equation}
where the coefficients $a_n$ and $b_n$ are explicitely given in \cite{KR}, and we recall $(x)_n:=x(x+1)\dots(x+n-1)$. We want to point out that our linear combination \eqref{formelineaire} gives a $q$-analogue of \eqref{GKR}. Indeed, for any odd positive integer $n$, $A=3$ and $r=1$, \eqref{formelineaire} is:
\begin{multline}\label{qG}
(q)_n\sum_{k\geq
1}(-1)^{k+1}q^{(2k-1)(n+1)/4}(1-q^{k+(n-1)/2})\frac{(q^{k-n},q^{k+n})_{n}}{(q^{k-1/2})_{n+1}^{3}}\\
=A_n(q)\beta_{\sqrt{q}}(2)+B_n(q),
\end{multline}
with
$$A_n(q):=\sum_{j=0}^n(-1)^jq^{1/2-j}(2d_{2,j,n}(q)+d_{3,j,n}(q))$$
and 
\begin{multline*}
B_n(q):=\sum_{s=1}^{3}\sum_{j=1}^n\sum_{k=1}^j(-1)^{j+k}\left(\frac{q^{k-j}d_{s,j,n}(q)}{(1-q^{k-1/2})^s}+\frac{q^{-k+j}d_{s,j,n}(1/q)}{(1-q^{-k+1/2})^s}\right)\\
-\frac{1}{2}\sum_{j=0}^n(-1)^jq^{1/2-j}d_{1,j,n}(q).
\end{multline*}
Multiplying \eqref{qG} by $(1-q^{1/2})^2$, then letting $q$ tend to $1$, we get by using \eqref{qanal}:
\begin{equation}\label{G}
-\frac{1}{2}\,n!\sum_{k\geq
1}(-1)^{k}\left(k+\frac{n-1}{2}\right)\frac{(k-n)_n(k+n)_{n}}{(k-1/2)_{n+1}^{3}}=\alpha_nG+\beta_n,
\end{equation}
where $\alpha_n:=\lim_{q\to1}A_n(q)$ and $\beta_n:=\lim_{q\to1}(1-q^{1/2})^2B_n(q)$. If the hypergeometric series on the left-hand side of \eqref{G} is multiplied by $-2$, we obtain the left-hand side of \eqref{GKR}. Moreover, by using the definition \eqref{Rn} of $R_n$, the identity \eqref{Alpha}, and \eqref{c2}, a direct calculation gives:
\begin{multline*}
\alpha_n=-2\sum_{j=0}^n(n-2j)\bi{n}{j}^3\bi{n+j-\frac{1}{2}}{n}\bi{2n-j-\frac{1}{2}}{n}\\
\times\left(\frac{1}{n-2j}+3H_j+H_{j-\frac{1}{2}}-H_{n+j-\frac{1}{2}}\right),
\end{multline*}
where for any positive integer $m$, $H_m:=\sum_{j=1}^{m}\frac{1}{j}$ is the $m$-th harmonic number, whose definition is extended to half-integers by $H_m:=\sum_{j=1}^{\lfloor m\rfloor+1}\frac{1}{m-j+1}\cdot$ This shows that for any odd positive integer $n$, $\alpha_n=-a_n/2$ (see \cite{KR}). Thus our linear combination \eqref{qG} is indeed a $q$-analogue of the one in \cite{KR}, as we necessarily have $\beta_n=b_n/2$. 

\subsection{A $q$-denominators Conjecture}

As in \cite{JM}, it seems possible to refine Lemma~4.5, by considering another common denominator to the $\hat{P}_{j,n}(q^2)$, having the form $\tilde{D}_n(q)=D_n(q)/\Delta_n(1/q)$. We formulate in this direction the following $q$-denominators Conjecture: 
\begin{conj}
Let $n$ and $A$ be odd positive integers, and set $r\in\N^*$ such that $A-2r>0$. For $\alpha=-A/4-r^2$, there exist real numbers $\beta$ and $\gamma$ only depending on $A$ and $r$ such that if we set:
\begin{equation*}\label{DenominateurcommuntildeDn}
\tilde{D}_n(q):=(A-1)!\,q^{\lfloor\alpha n^2+\beta
n+\gamma\rfloor}\varphi_n(1/q)^{2r}\,d_{2n}(1/q)^{A-1},
\end{equation*}
then we get:
$$\tilde{D}_n(q)\hat{P}_{j,n}(q^2)\in\Z\left[\frac{1}{q}\right]\;\;\forall
j\in\{0,2,4,\dots,A-1\}.$$
\end{conj}
In our opinion, this Conjecture should be solved by using transformation formulae for basic hypergeometric series, together with arithmetical techniques similar to those used in \cite{JM}. Proving this conjecture would imply for $1/q\in\Z\setminus\{-1;1\}$ and any odd integer $A\geq 3$:
\begin{equation*}
\dim_{\Q}\left(\Q+\Q\beta_q(2)+\dots+\Q\beta_q(A-1)\right)\geq
g(A),
\end{equation*}
where
$$
g(A)=\max_{r\in\N\atop 1\leq r<A/2}g(r;A)\;\;\;\;\mbox{and}\;\;g(r;A):=\frac{4rA+A-4r^2}{\left(\frac{48}{\pi^2}+2\right)A+8r^2-\frac{48}{\pi^2}+\frac{16r}{3}}\cdot
$$
Unfortunately, and unlike the case of $\zeta_q$ in \cite{JM}, solving this conjecture would not give a refinement of the quantitative version in Corollary~1.3. Indeed, although $g(A)>f(A)$ for all $A$, we have $1>g(19)\simeq0.988>f(19)\simeq0.973$ and $g(21)\simeq1.042>f(21)\simeq1.028>1$.


\vspace{1cm}

\noindent Fr\'ed\'eric Jouhet,\\
Universit\'e de Lyon, Universit\'e Lyon I, \\
CNRS, UMR 5208 Institut Camille Jordan,\\
B\^atiment du Doyen Jean Braconnier,\\
43, bd du 11 Novembre 1918, 69622 Villeurbanne Cedex, France \\
\texttt{jouhet@math.univ-lyon1.fr}

\vspace{0.5cm}

\noindent Elie Mosaki,\\
Universit\'e de Lyon, Universit\'e Lyon I, \\
CNRS, UMR 5208 Institut Camille Jordan,\\
B\^atiment du Doyen Jean Braconnier,\\
43, bd du 11 Novembre 1918, 69622 Villeurbanne Cedex, France \\
\texttt{mosaki@math.univ-lyon1.fr}


\small
\begin{thebibliography}{99}
\bibitem{Ap} \textsc{Ap\'ery} (R.),
\emph{Irrationalit\'e de $\zeta(2)$ et $\zeta(3)$}, Ast\'erisque
{\bf 61} (1979), 11-13.
\bibitem{BV} \textsc{Bundschuh} (P.) and \textsc{V\"a\"an\"anen} (K.),
\emph{Arithmetical investigations of a certain infinite product},
Compositio Math. {\bf 91}.2 (1994), 175-199.
\bibitem{BZ} \textsc{Bundschuh} (P.) and \textsc{Zudilin} (W.),
\emph{Rational approximation to a $q$-analogue of $\pi$ and some other $q$-series}, proceedings of the 70th birthday conference in honour of W. M. Schmidt(Vienna, November 2003), Vienna, Springer-Verlag (2007), 17 pages.
\bibitem{BZ2} \textsc{Bundschuh} (P.) and \textsc{Zudilin} (W.), \emph{Irrationality measures for certain q-mathematical constants}, Math. Scand.  {\bf 101}.1 (2007), 104--122.
\bibitem{GR} \textsc{Gasper} (G.) and \textsc{Rahman} (M.),
\emph{Basic Hypergeometric Series}, 2nd Edition,  Encyclopedia of
mathematics and its applications, Vol. {\bf 96}, Cambridge
University Press, Cambridge, 2004.
\bibitem{JM} \textsc{Jouhet} (F.) and \textsc{Mosaki} (E.), \emph{Irrationalit\'e aux entiers impairs positifs d'un $q$-analogue de la fonction z\^eta de Riemann}, preprint.
\bibitem{Ko} \textsc{Koblitz} (N.), \emph{Introduction to elliptic curves and modular forms}, Graduate Texts in Math. No. 97, Springer-Verlag, 1984. Second edition, 1993.
\bibitem{KR} \textsc{Krattenthaler} (C.) and \textsc{Rivoal} (T.), \emph{On a linear form for Catalan's constant}, South East Asian J. Math. Math. Sci. 6.2  (2008), 3-15.
\bibitem{KRZ} \textsc{Krattenthaler} (C.), \textsc{Rivoal} (T.) and \textsc{Zudilin} (W.), \emph{S\'eries hyperg\'eom\'etriques basiques, $q$-analogues
des valeurs de la fonction z\^eta et s\'eries d'Eisenstein}, J.
Inst. Jussieu {\bf 5}.1 (2006), 53-79.
\bibitem{Ne} \textsc{Nesterenko} (Yu. V.),  \emph{On the linear independance of
numbers}, (in russian) Vest. Mosk. Univ., Ser. I, no. 1 (1985), 46-54; english transl. in Mosc. Univ. Math. Bull. {\bf 40}.1 (1985),
69-74.
\bibitem{Ne2} \textsc{Nesterenko} (Yu. V.),  \emph{Modular functions and
transcendance questions}, (in russian) Math. Sb. {\bf 187}.9 (1996), 65-96; english transl. in Sb. Math. {\bf 187}.9 (1996), 1319-1348.
\bibitem{R} \textsc{Rivoal} (T.),
\emph{Nombres d'Euler, approximants de Pad\'e et constante de Catalan}, Ramanujan J. {\bf 11} (2006), 199-214.
\bibitem{RZ} \textsc{Rivoal} (T.) and \textsc{Zudilin} (W.),
\emph{Diophantine properties of numbers related to Catalan's constant}, Math. Annalen {\bf 326}.4 (2003), 705-721.
\bibitem{St} \textsc{Stanley} (R. P.), \emph{Enumerative Combinatorics}, Vol. 1, Cambridge University Press, Cambridge, 1998.
\bibitem{VA} \textsc{Van Assche} (W.),  \emph{Little $q$-Legendre polynomials and irrationality of certain Lambert series}, Ramanujan J. {\bf 5}.3 (2001),
295-310.
\end{thebibliography}
\end{document}